\documentclass[reqno]{amsart}
\usepackage{mathrsfs}
\usepackage{amsthm}
\usepackage{amsmath}
\usepackage{amssymb}
\usepackage{amscd}
\usepackage[pdftex]{graphicx}
\usepackage{hyperref}
\usepackage{color}
\usepackage{comment}
\usepackage{pstricks}
\definecolor{omegagray}{gray}{0.93}
\definecolor{mygray}{gray}{0.55}
\def\beq{\begin{equation}}
\def\eeq{\end{equation}}
\def\ba{\begin{array}}
\def\ea{\end{array}}

\def\C{\mathbb C}

\allowdisplaybreaks

\numberwithin{equation}{section}
\newenvironment{abs}{\textbf{Abstract}\mbox{  }}{ }
\newenvironment{key}{\textbf{Keywords}\mbox{  }}{ }
\newtheorem*{theorem*}{Theorem A}
\newtheorem{theorem}{Theorem}[section]
\newtheorem{lemma}[theorem]{Lemma\hspace*{0.1em}}

\newtheorem{proposition}[theorem]{Proposition\hspace*{0.1em}}
\newtheorem{corollary}[theorem]{Corollary\hspace*{0.1em}}
\newtheorem{remark}[theorem]{Remark\hspace*{0.1em}}
\renewenvironment{proof}{\textbf{Proof.}}{\hspace*{0.1em}\hfill$\Box$}

\begin{document}
\title[Onofri trace inequality ]{\textbf{
Classification of solutions to the Liouville equation with a nonlinear Robin boundary condition on the unit disk}}

\author[J. Dou,   Y. Hu ]{Jingbo Dou, Yunyun Hu, Keqing Peng}

\address{Jingbo Dou, School of Mathematics and Statistics,
Shaanxi Normal University,
Xi'an, 710119, P. R. China}
\email{jbdou@snnu.edu.cn}
\medskip
\address{Yunyun Hu, School of Mathematics and Statistics,
Shaanxi Normal University,
Xi'an, 710119, P. R. China}
\email{yhu@snnu.edu.cn}
\medskip

\address{Keqing Peng, School of Mathematics and Statistics,
Shaanxi Normal University,
Xi'an, 710119, P. R. China}
\email{peng\_keqing@163.com}
\medskip

\date{}

\maketitle

% ----------------------------------------------------------------
\noindent
\begin{abs}
In this paper, we study the nonlinear boundary value problem
\begin{equation*}
 \begin{cases}
-\Delta u=e^{2u},& \mbox{in }  {\mathbb{D}},\\
\frac{\partial u}{\partial\nu}+\lambda=e^u ,& \mbox{on }  {\mathbb{S}^{1}},
 \end{cases}
  \end{equation*}
where $\mathbb D$ is the unit disk, $\lambda$ is a constant and $\nu$ denotes the outer unit normal on $\mathbb S^1$. For $0<\lambda\le2$, we establish a complete classification of smooth solutions. For $2<\lambda\leq3$, we prove a dichotomy between radial solutions and nonradial solutions. We develop a Hardy-Wronskian boundary rigidity method that transforms the nonlinear boundary value problem into a spectral rigidity problem for normalized holomorphic frames.
This method exploits the complex-analytic structure of the Liouville equation with a Robin boundary condition and provides a new rigidity framework for related two-dimensional elliptic boundary value problems.
\end{abs}\\
\begin{key}
Nonlinear elliptic equations, classification, normalized holomorphic lift, Hardy space, developing map.
\end{key}\\
\textbf{Mathematics Subject Classification(2020).}
35J25, 35R01, 58J05, 30H10
\indent

%---------------------------------------------------------------------------------
\section{\textbf{Introduction}\label{Section 1}}

Let $(\Sigma, g)$ be a compact surface with boundary, and let
$$g_u=e^{2u}g$$
be a conformal metric to $g$.
The Gaussian curvature $K_{g_u}$ and the boundary geodesic curvature $\kappa_{g_u}$
then satisfy
\begin{equation}\label{gubvv}
\begin{cases}
-\Delta_{g}u+K_g=K_{g_u}e^{2u},& \mbox{in }  \Sigma,\\
\frac{\partial u}{\partial\nu}+\kappa_g=\kappa_{g_u}e^{u},& \mbox{on }  {\partial \Sigma}.
\end{cases}
\end{equation}
The problem of prescribing these curvatures within a conformal class is a classical theme in geometric analysis.
On closed surfaces, the prescribed Gaussian curvature problem was systematically studied in the seminal works of \cite{KW1971, KW1974, CY1988}.

In the last decades, various versions of problem \eqref{gubvv} have been extensively studied from both analytic and geometric viewpoints. Chang and Yang \cite{CY1988} investigated the case $\kappa_{g_u} = 0$. The scalar-flat case $K_{g_u}= 0$  was studied in \cite{CL1996,LL2005,LH2005}.
The problems with constant prescribed curvatures $K_{g_u}$ and $\kappa_{g_u}$ have also been considered. Using a parabolic flow, Brendle \cite{B2002}
proved that this problem admits a solution for some constant curvatures.
When $\Sigma$ is a disk or an annulus, the classification of solutions are given by using complex analysis techniques in \cite{HW2006, J2012}.
The case of the half-plane has also been studied see \cite{LZ1995, Z2003}.

In particular, Wang \cite{W2017} established existence and classification
results for a conformal metric with zero scalar curvature
and constant mean curvature on the boundary.
Wang \cite{W2017} studied the following equation
\begin{equation}\label{WfOnofri-E-L}
 \begin{cases}
-\Delta_g u=0,& \mbox{in }  {\Sigma},\\
\frac{\partial u}{\partial\nu}+\lambda=e^u ,& \mbox{on }  {\partial \Sigma}.
 \end{cases}
 \end{equation}
Using a pointwise analysis and the strong maximum principle, the author proved that for $0<\lambda<1$, $u$ is a constant. When $\lambda=1$ and $u$ is non-constant, $\Sigma$ is isometric to the unit disk $ \overline{\mathbb{B}^2}$ and $u$ must be of the form
\begin{align*}
u=\log
\frac{1-|a|^2}{1+|a|^2 |x|^2 -2Re(x\bar{a})}
\end{align*}
for some $a\in\mathbb{B}^2$. On the unit disk, Wang further obtained a complete classification for all $\lambda>0$ using a
Fourier-Hilbert transform approach, which is specific to the harmonic setting.

In the case $\lambda=1$ on $\overline{\mathbb{B}^2}$, equation \eqref{WfOnofri-E-L} can be seen as the Euler-Lagrange equation of the Onofri trace inequality.  It has been studied by many authors
by various different methods, see \cite{HW2006,OPS1988,Z2003}.

When $\Sigma$ is the unit disk, curvature prescription problems naturally lead to Liouville type elliptic equations with exponential nonlinearities.
Let $(\Sigma,g)=({\mathbb{D}},g_0)$, where $g_0=|dz|^2$. The conformal
metric $g_u=e^{2u}g$ becomes
$g_u=e^{2u}g_0$. Then the conformal curvature transformation law \eqref{gubvv} simplifies to
\begin{equation}\label{regubvv}
\begin{cases}
-\Delta u=K_{g_u}e^{2u},& \mbox{in }  {\mathbb{D}},\\
\frac{\partial u}{\partial\nu}+1=\kappa_{g_u}e^{u},& \mbox{on }  \mathbb{S}^{1},
\end{cases}
\end{equation}
where $\nu$ is the outer unit normal on $\mathbb{S}^{1}$.
Cruz-Bl\'{a}zquez and Ruiz \cite{BR2018} studied the problem of prescribing the Gaussian curvature and the geodesic curvature on the disk.
More precisely, under the normalization $g_u=e^{2u}g$, their problem can be written as \eqref{regubvv}. Under suitable symmetry assumptions, they proved the existence of solution in a variational framework.

Recently, Dou and Xu \cite{DX2025} studied the semilinear elliptic equation
\begin{equation}\label{DX-EL}
 \begin{cases}
-\Delta u=e^{2u},& \mbox{in }  {\mathbb{B}^2},\\
\frac{\partial u}{\partial\nu}+\lambda=0 ,& \mbox{on }  {\mathbb{S}^{1}},
 \end{cases}
\end{equation}
where $\mathbb{B}^2\subset \mathbb R^2$ is a unit ball.
They established a complete classification of solutions to \eqref{DX-EL} for $\lambda\in(0,1]$.  They proved that for $0<\lambda<1$, $u$ is given by
\begin{align*}
u=\log
\frac{2R}{R^2|x|^2+1},
\end{align*}
where $R=\sqrt{\frac{\lambda}{-\lambda+2}}$.
If $\lambda=1$, then $u$ must be of form
\begin{align*}
u=\log
\frac{2(1-|\zeta|^2)}{(1+|\zeta|^2)(|x|^2-\frac{4}{1+|\zeta|^2}\langle\zeta,x\rangle+1)},
\end{align*}
for some $\zeta\in\mathbb{B}^2$.
Their proof relies on some integral identities and Obata identity.
They also showed that equation \eqref{DX-EL} is the Euler-Lagrange equation of an Onofri-type inequality with Neumann boundary conditions. In \cite{DX2025}, a fundamental open problem remained: can one classify all
solutions to \eqref{DX-EL} for $1<\lambda<2$?

For more results concerning the exponential Neumann problem and related geometric boundary problems, see \cite{LZ1995, GM2009,GJM2012, Ruiz2024} and the references therein. These problems have been investigated via moving spheres, complex-analytic methods, variational techniques, and blow-up analysis.

%For more results concerning the existence and classification of positive solution for the
%related boundary value problem, please refer to \cite{ BWZ2014,SSR2024,SSR2026} and
%the references therein.

In most previous literature, classification results on the disk mainly concern either the scalar flat boundary problem or the Neumann type curvature equation. In contrast, not much is known for the problem coupling an interior Liouville term with a nonlinear boundary curvature term.
In this paper, we investigate the following
boundary value problem on the unit disk
\begin{equation}\label{Onofri-E-L}
 \begin{cases}
-\Delta u=e^{2u},& \mbox{in }  {\mathbb{D}},\\
\frac{\partial u}{\partial\nu}+\lambda=e^u ,& \mbox{on }  {\mathbb{S}^1},
 \end{cases}
\end{equation}
where $\lambda\in \mathbb R$.

The Robin condition in \eqref{Onofri-E-L} is a special nonlinear
boundary curvature law. Indeed, for the conformal metric $g_u=e^{2u}|dz|^2$,
the boundary geodesic curvature is given by
$$\kappa_{g_u}=e^{-u}(1+\partial_\nu u)=1+(1-\lambda)e^{-u}.$$
Hence, the case $\lambda=1$ corresponds to constant geodesic
curvature $\kappa_{g_u}\equiv1$, while $\lambda\neq1$ yields the
solution-dependent boundary curvature.

Throughout the paper, we write
$$\mathbb{D}=\mathbb{B}^2=\{z=x_1+ix_2 \in\mathbb{C}:|z|<1\},\quad\ \partial\mathbb D=\mathbb S^1,$$
$$\operatorname{Aut}(\mathbb D)=\{\psi: \mathbb D\rightarrow \mathbb D\mid \psi\ \mbox{and}\ \psi^{-1}\  \mbox{are\ holomorphic } \}.$$

Our first main result provides a complete classification of solutions in the non-conformally invariant range $\lambda\leq2$, $\lambda\neq1$.

\begin{theorem}\label{radial-rigidity}
Let $\lambda\leq 2$, $\lambda\neq1$, and
$u\in C^\infty(\overline{\mathbb D})$ be a solution of
\eqref{Onofri-E-L}. The following results hold.

 $(i)$ If  $\lambda\leq0$, then equation \eqref{Onofri-E-L} has no solution.

 $(ii)$  If $0<\lambda\leq2$ with $\lambda\neq1$, then $u$ must be of the form
\begin{equation}\label{R-class-large}
u(z)=\log\frac{2R}{1+R^2|z|^2},
\end{equation}
where $R>0$ is the unique positive solution of
\begin{equation}\label{R-disc-range}
(\lambda-2)R^2-2R+\lambda=0.
\end{equation}
\end{theorem}

The coupling of the interior Liouville nonlinearity \(e^{2u}\) with
the nonlinear boundary term \(e^u\) produces a boundary structure
that is absent from both the scalar-flat problem and the
Neumann-type curvature problem. Consequently, in the supercritical
case \(\lambda>1\), several classical tools, including
Obata-type identities, maximum-principle arguments and the sphere
covering inequality \cite{GM2018}, are no
longer directly available for
\eqref{Onofri-E-L}.

In this paper, we develop a Hardy-Wronskian method that is different from
the method of moving planes (or methods), variational techniques and standard blow-up arguments.
We first lift the
positive function $e^{-u}$ to a normalized holomorphic \(\C^2\)-valued map.  The quotient of its
components is the classical developing map, while the corresponding scalar holomorphic potential is, up to a
constant factor, its Schwarzian derivative (see Remark \eqref{developing-map} below for details).  We then decompose the lift into a scalar outer factor and a
column-inner factor.
The boundary equation yields some
identities for the Taylor coefficients of these factors. For
$0<\lambda\leq2$, the coefficient relations force the
outer factor to be constant and hence the lift to be affine. For $0<\lambda\leq2$ with $\lambda\neq1$, it
follows that $u$ is radial and has the explicit form
\eqref{R-class-large}.

The factorization underlying our method falls within the framework of the scalar Hardy space theory
\cite{D1970,G2007} and  the Beurling--Lax theorem for vector-valued Hardy spaces
\cite{B1949,L1959,RR1985,NN2002}.  Similar connections between Liouville type equations and inner functions have been established by Kraus and Roth \cite{KR2008} in a different curvature setting.
Our method may also be applicable to constant-curvature Liouville-type equations, normalized holomorphic ODEs, and related geometric equations with spinorial representations.

\begin{remark}
The proof of Theorem \ref{radial-rigidity} requires the
nonvanishing of \(1-\lambda\) in the comparison of the first Fourier
modes. At the conformally invariant value \(\lambda=1\), this
comparison becomes degenerate and no longer yields the required
rigidity. This is consistent with the existence of nonradial solutions
generated by M\"obius transformations.
\end{remark}

We now state our second main result in the conformally invariant case $\lambda=1$.

\begin{theorem}\label{critical-classification}
Let $\lambda=1$ and $u\in C^\infty(\overline{\mathbb D})$ be a solution of
\eqref{Onofri-E-L}.
Then there exists $\psi\in\operatorname{Aut}(\mathbb D)$ such that
\begin{equation}\label{p-critical-classification}
u(z)
=
\log\frac{2R|\psi'(z)|}
{1+R^2|\psi(z)|^2},
\end{equation}
where \(R=\sqrt2-1\).
Conversely, every function of the form
\eqref{p-critical-classification} with
$\psi\in\operatorname{Aut}(\mathbb D)$ is a solution of
\eqref{Onofri-E-L}.

Moreover, the automorphism $\psi$ is unique up to a rotation: if $\psi_1,\psi_2\in\operatorname{Aut}(\mathbb D)$
determine the same function $u$, then
\[
\psi_2=e^{i\theta}\psi_1
\]
for some $\theta\in\mathbb R$.
\end{theorem}

Different from Theorem \ref{radial-rigidity}, we first construct a Hardy-spinor global developing map. In the case $\lambda=1$, the vanishing of the projective differential forces the developing map to be a M\"{o}bius transformation. The boundary condition then determines its image as a spherical disk, which, after normalization, yields the disk automorphism classification.

\begin{remark}
The representation \eqref{p-critical-classification} is equivalent to the standard spherical bubble representation. In particular, all solutions in the conformally invariant case \(\lambda=1\) can be written explicitly as
\begin{align*}
u=\log
\frac{2\mu}{|x-x_0|^2+\mu^2}
\end{align*}
for some $\zeta\in\mathbb{B}^2$, where $x_0=\frac{2\zeta}{|\zeta|^2-1+\sqrt{2(1+|\zeta|^4)}}$ and $\mu=\frac{1-|\zeta|^2}{|\zeta|^2-1+\sqrt{2(1+|\zeta|^4)}}$.
\end{remark}

For $2<\lambda\leq3$, we establish the following dichotomy for all solutions.

\begin{theorem}\label{pde-quadratic}
Let $2<\lambda\leq3$  and
$u\in C^\infty(\overline{\mathbb D})$ be a solution of \eqref{Onofri-E-L}.  Then exactly one of the following holds.

$(i)$ \(u\) is radially symmetric and has the form \eqref{R-class-large}, where \(R>0\) satisfies \eqref{R-disc-range}.

$(ii)$ \(u\) is nonradial and
\begin{equation*}\label{pde-q0-nonzero}
 u_z(0)^2-u_{zz}(0)\ne0.
\end{equation*}

Moreover, if \(1+\sqrt2<\lambda\leq3\), every solution, if it exists, satisfies alternative \textup{(ii)}.
\end{theorem}

As an immediate consequence of the radial classification,
we obtain the following result for all \(\lambda>0\).

\begin{corollary}\label{radial-all-lambda}
Let $\lambda>0$.  Then the following assertions hold.

$(i)$ If $0<\lambda<2$,  problem \eqref{Onofri-E-L}
admits a unique radial solution of the form \eqref{R-class-large}, where
\[
R=\frac{\lambda}
{1+\sqrt{1+2\lambda-\lambda^2}}.
\]

$(ii)$ If $\lambda=2$, problem \eqref{Onofri-E-L}
admits a unique radial solution of the form \eqref{R-class-large}, where $R=1$.

$(iii)$ If $2<\lambda<1+\sqrt2$,  problem \eqref{Onofri-E-L}
admits exactly two radial solutions of the form \eqref{R-class-large}, where
\begin{equation*}\label{eq-Rpm}
 R_\pm
 =
 \frac{1\pm\sqrt{1+2\lambda-\lambda^2}}
 {\lambda-2}.
\end{equation*}

$(iv)$ If $\lambda=1+\sqrt2$, problem \eqref{Onofri-E-L}
admits a unique radial solution of the form \eqref{R-class-large}, where
\[
 R=1+\sqrt2.
\]

$(v)$ If $\lambda>1+\sqrt2$, there is no radial solution.
\end{corollary}

The rest of this paper is organized as follows. In Section 2, we first construct a normalized
holomorphic lift, then establish the corresponding
holomorphic rigidity result. Finally, we derive the classification result in the case $0<\lambda\leq2$, $\lambda\neq1$. In Section 3, we prove Theorem \ref{critical-classification} by constructing a global developing map and analyzing its projective differential. Section 4 is devoted to the proof of Theorem \ref{pde-quadratic} and Corollary \ref{radial-all-lambda}.

\section{Classification of solutions for $\lambda\leq 2$, $\lambda\neq1$}

In this section, we give the proof of Theorem \ref{radial-rigidity}. We first construct a normalized holomorphic lift, which is equivalent to the equation \eqref{Onofri-E-L}. Then we prove the holomorphic rigidity result. Finally, we give a classification of solutions in the range $\lambda\leq2$, $\lambda\neq1$.

We now present our results and introduce some notations that will be used in the following. Using Euler's formula, we write
$$ z=re^{i\theta}\ \ \mbox{with}\ r\geq0,\quad\ r=|z|=\sqrt{x_1^2+x_2^2}.$$
Define the complex derivatives
\[
 \partial_z=\frac12(\partial_{x_1}-i\partial_{x_2}),
 \qquad
 \partial_{\bar z}=\frac12(\partial_{x_1}+i\partial_{x_2}),
 \qquad
 \Delta=4\partial_z\partial_{\bar z},
\]
and let $\partial_r$ denote the radial derivative in polar coordinates, i.e.,
\[
\partial_r=e^{i\theta}\partial_z+e^{-i\theta}\partial_{\bar{z}},
\]
where $\bar{z}$ is the conjugate of $z$.
The Hermitian product is
\[
\langle \xi, \eta \rangle=\xi_1\overline{\eta_1}+\xi_2\overline{\eta_2},
\]
where vectors $\xi,\eta\in \C^2$.

Set
$$\mathcal{O} (\mathbb D;\mathbb C^2)=\{f: \mathbb D\rightarrow\mathbb C^2\mid f\ is\ holomorphic\ in\ \mathbb D\}$$
and
$$A^{\infty}(\mathbb D;\mathbb C^2)=\mathcal{O} (\mathbb D;\mathbb C^2)\cap C^{\infty}(\overline{\mathbb D};\mathbb C^2).$$
Set $a=\lambda-1$ and
\begin{equation*}\label{new-def}
p=e^{-u},\quad \ Y=(y_1,y_2)^{T},
\end{equation*}
$$\phi=u_{zz}-u_z^2.$$
For simplicity, we write $ u_z=\partial_zu$ and $ u_{\bar{z}}=\partial_{\bar{z}}u$ in the following.

\subsection{The normalized holomorphic lift}

In this subsection, we construct a normalized holomorphic lift. This lift is equivalent to the
equation \eqref{Onofri-E-L}.

%Define
%$$(\kappa_{\bar{g}_\phi}|_{\Gamma_D})_+=\max \{\kappa_{\bar{g}_\phi}|_{\Gamma_D},0\},\ \quad\ (\kappa_{\bar{g}_\phi}|_{\Gamma_D})_-=\max \{-\kappa_{\bar{g}_\phi}|_{\Gamma_D},0\}.$$

The equation \eqref{Onofri-E-L} yields some identities as follows.

\begin{lemma}\label{p-identities}
Let $u\in C^\infty(\overline{\mathbb D})$  be the solution of \eqref{Onofri-E-L}.  Then
\begin{equation}\label{p-diff-identity}
 pp_{z\bar z}-p_zp_{\bar z}=\frac14,
\end{equation}
\begin{equation}\label{diff-p-zz}
 p_{zz}+\phi p=0
 \quad \mbox{and}\quad
 \phi\in A^{\infty}(\mathbb D).
\end{equation}
\end{lemma}

\begin{proof}
Using \eqref{Onofri-E-L} yields
\[
 u_{z\bar z}=-\frac14e^{2u}.
\]
Then we have
\[
 p_z=-u_zp,\quad p_{\bar z} =-u_{\bar z} p
\]
and
\[
 p_{z\bar z}=\partial_{\bar{z}}(p_z)
 =\partial_{\bar z}(-u_zp)
 =-u_{z\bar z}p-u_zp_{\bar z}
 =(-u_{z\bar z}+u_zu_{\bar z})p.
\]
Therefore,
\[
 \begin{aligned}
 pp_{z\bar z}-p_zp_{\bar z}
 &=p^2(-u_{z\bar z}+u_zu_{\bar z})-u_zu_{\bar z}p^2=-u_{z\bar z}p^2=\frac14,
 \end{aligned}
\]
which prove \eqref{p-diff-identity}.

A direct calculation gives
\[
 p_{zz}=\partial_{z}(p_z)=\partial_z(-u_zp)
 =-u_{zz}p-u_zp_z
 =(-u_{zz}+u_z^2)p=-\phi p
\]
and
\[
 \begin{aligned}
 \partial_{\bar z}\phi
 &=u_{zz\bar z}-2u_zu_{z\bar z}\\
 &=\partial_z\!\left(-\frac14e^{2u}\right)
   -2u_z\!\left(-\frac14e^{2u}\right)\\
 &=-\frac12e^{2u}u_z+\frac12e^{2u}u_z=0.
 \end{aligned}
\]
Thus \(\phi\) is holomorphic. Moreover, it follows from $u\in C^\infty(\overline{\mathbb D})$ that $\phi\in A^{\infty}(\mathbb D)$.
\end{proof}

The next proposition is normalized holomorphic lift. The main ingredient we use is the standard existence of a fundamental system for a linear holomorphic ODE on a simply connected domain, see \cite{Hille1976}.
\begin{proposition}\label{normalized-lift}
Let $u\in C^\infty(\overline{\mathbb D})$ be a solution of \eqref{Onofri-E-L}. Then there exists $Y\in A^{\infty}(\mathbb D)$
such that
\begin{equation}\label{lift-normalization}
 \det(Y',Y)=1,
 \qquad
 p=\frac{|Y|^2}{2}.
\end{equation}
The lift $Y$ is unique up to left multiplication by a constant matrix in \(SU(2)\).
\end{proposition}

\begin{proof}
Let \(y_1,y_2\) be a fundamental system for
\begin{equation}\label{scalar-ode-T}
 y''+\phi y=0\quad \mbox{in}\ \mathbb D.
\end{equation}
Note that the Wronskian
$$W(y_1,y_2)=\det(Y',Y)=y'_1y_2-y'_2y_1.$$
Equation \eqref{scalar-ode-T} gives
$$W_z(y_1,y_2)=y''_1y_2-y''_2y_1=0,$$
so the Wronskian is constant. Without loss of generality, we may normalize the fundamental system such
that, for
\[
 Y_0=(y_1,y_2)^T,
\]
one has
\begin{equation}\label{Y0-Wronskian}
 \det(Y_0',Y_0)=1.
\end{equation}
Define
\[
 M(z)=
 \begin{pmatrix}
 y_1(z)&y_2(z)\\
 y_1'(z)&y_2'(z)
 \end{pmatrix}.
\]
From \eqref{Y0-Wronskian}, we immediately see that \(\det M=-1\). Hence \(M\) is invertible in \(\mathbb D\).

Define
\[
 V=\binom{p}{p_z},
 \qquad
 A(z)=
 \begin{pmatrix}
 0&1\\
 -\phi(z)&0
 \end{pmatrix}.
\]
From Lemma \ref{p-identities}, it follows that
$$V_z=\binom{p_z}{p_{zz}}=\binom{p_z}{-\phi p}=AV$$
and
$$M_z=\begin{pmatrix}
 y'_1(z)&y'_2(z)\\
 y_1''(z)&y_2''(z)
  \end{pmatrix}=\begin{pmatrix}
 y'_1(z)&y'_2(z)\\
 -\phi y_1(z)&-\phi y_2(z)
  \end{pmatrix}= AM.$$
Set
\[
 c=\binom{c_1}{c_2}=M^{-1}V.
\]
Then we have
\[
 c_z
 =-M^{-1}M_zM^{-1}V+M^{-1}V_z
 =-M^{-1}AV+M^{-1}AV=0.
\]
Thus \(c\) is anti-holomorphic, and
\begin{equation}\label{p-y-c}
 p(z,\bar z)=y_1(z)c_1(\bar z)+y_2(z)c_2(\bar z).
\end{equation}

Since \(p\) is real-valued function, the conjugate of \eqref{diff-p-zz} is
\begin{equation}\label{conjugate-p-ode}
 p_{\bar z\bar z}+\bar{\phi}p=0.
\end{equation}
Substituting \eqref{p-y-c} into \eqref{conjugate-p-ode} gives
\begin{equation}\label{sys-a-ide}
 y_1(z)a_1(\bar z)+y_2(z)a_2(\bar z)=0,
\end{equation}
where
\[
 a_j(\bar z):=c_j''(\bar z)+\bar{\phi}\,c_j(\bar z)\quad\ j=1,2.
\]
Each \(a_j\) is anti-holomorphic.  Differentiating \eqref{sys-a-ide} with respect to
\(z\) yields
\[
 y_1'(z)a_1(\bar z)+y_2'(z)a_2(\bar z)=0.
\]
The invertibility of \(M(z)\) then implies that
\[
 a_1=a_2=0.
\]
Hence each \(c_j\) solves
$$c_j''(\bar z)+\bar{\phi}\,c_j(\bar z)=0,\quad\ j=1,2.$$
 and is a constant linear combination of
\(\overline{y_1}\) and \(\overline{y_2}\).  Consequently, there is a constant matrix
\(H\in\C^{2\times2}\) such that
\begin{equation}\label{eq-p-H}
 p=Y_0^*HY_0,
\end{equation}
where $Y_0^*$ denotes the conjugate transpose of $Y_0$.

We next show that \(H\) is Hermitian.  The reality of \(p\) in \eqref{eq-p-H} gives
\begin{equation*}\label{Her-Y-H}
 Y_0^*(H-H^*)Y_0=0.
\end{equation*}
Differentiating this identity with respect to \(z\) and \(\bar z\), we have
\[
 Y_0^*(H-H^*)Y_0'=0,
 \qquad
 Y_0'^*(H-H^*)Y_0=0,
 \qquad
 Y_0'^*(H-H^*)Y_0'=0.
\]
It follows that
\[
 \begin{pmatrix}Y_0&Y_0'\end{pmatrix}^*
(H-H^*)
 \begin{pmatrix}Y_0&Y_0'\end{pmatrix}=0.
\]
By \eqref{Y0-Wronskian}, the matrix \((Y_0,Y_0')\) is invertible, and hence
\(H=H^*\).

Differentiating \eqref{eq-p-H} gives the matrix identity
\begin{equation}\label{matrix-H-identity}
 \begin{pmatrix}
 p&p_z\\
 p_{\bar z}&p_{z\bar z}
 \end{pmatrix}
 =
 \begin{pmatrix}Y_0&Y_0'\end{pmatrix}^*
 H
 \begin{pmatrix}Y_0&Y_0'\end{pmatrix}.
\end{equation}
Using Lemma \ref{p-identities} and $\det(Y_0,Y_0')=1$, we derive
\begin{equation*}\label{eq-det-H}
 \det H=\frac14.
\end{equation*}
The matrix on the left of \eqref{matrix-H-identity} has positive leading entry $p$ and determinant $\frac{1}{4}$, and is therefore positive definite. The identity \eqref{matrix-H-identity}, together with the invertibility of \((Y_0,Y_0')\), show that
\(H\) is positive definite.  Let
\[
 B=(2H)^{1/2}
\]
be the positive Hermitian square root.  Then
\[
 H=\frac12B^*B,
 \qquad
 \det B=\sqrt{\det(2H)}=1.
\]
Setting \(Y=BY_0\), a direct calculation yields
\[
 \frac{|Y|^2}{2}
 =\frac12Y_0^*B^*BY_0
 =Y_0^*HY_0=p,
\]
\[
 \det(Y',Y)
 =\det B\,\det(Y_0',Y_0)=1.
\]
This proves \eqref{lift-normalization}.

We now prove the uniqueness of $Y$. Let \(\widetilde Y\) be another normalized lift.  Since
\(\det(\widetilde Y',\widetilde Y)=1\), there exist holomorphic \(\alpha,\beta\) such that
\[
 \widetilde Y''=\alpha\widetilde Y'+\beta\widetilde Y.
\]
Differentiating
\(\det(\widetilde Y',\widetilde Y)=1\) gives
\[
 0=\det(\widetilde Y'',\widetilde Y)=\alpha.
\]
Therefore, \(\widetilde Y''=\widetilde q\,\widetilde Y\) for a holomorphic scalar function
\(\widetilde q\).  Since \(p=\frac{|\widetilde Y|^2}{2}\),
\[
 p_{zz}=\frac12\widetilde Y^*\widetilde Y''=\widetilde q p.
\]
Together with \eqref{lift-normalization} gives
\[\widetilde q=-\phi.\]
Applying the same argument to
\(Y\) implies that \(Y\) and \(\widetilde Y\) are fundamental systems for the same scalar equation.
Therefore,
\[
 \widetilde Y=CY
\]
for a constant \(C\in SL(2,\C)\).

The equality \(|\widetilde Y|=|Y|\) yields
\begin{equation}\label{Un-det-Y}
 Y^*(C^*C-I)Y=0.
\end{equation}
Differentiating \eqref{Un-det-Y} gives
\[
 Y^*(C^*C-I)Y'=0,
 \qquad
 Y'^*(C^*C-I)Y=0,
 \qquad
 Y'^*(C^*C-I)Y'=0.
\]
Since \((Y,Y')\) is invertible, we have \(C^*C=I\). Combining this with
\(\det C=1\) implies \(C\in SU(2)\).
\end{proof}

Next we prove that the lift $Y$ satisfies equation \eqref{Onofri-E-L}. Combining with Proposition \ref{normalized-lift} gives the desired equivalence between the lift $Y$ and equation \eqref{Onofri-E-L}.

\begin{proposition}\label{prop-reverse-lift}
Let
\(
 Y\in A^\infty(\mathbb D;\mathbb C^2)
\)
satisfy \(\det(Y',Y)=1\) and
\begin{equation}\label{boundary-p-Y}
 \partial_r|Y|^2=\lambda|Y|^2-2\quad\text{on }\ \mathbb S^1.
\end{equation}
For
\[
 p=\frac{|Y|^2}{2},
\]
we have \(p>0\) on \(\overline{\mathbb D}\), and \(u=-\log p\) satisfies equation \eqref{Onofri-E-L}.
\end{proposition}

\begin{proof}
The normalization \(\det(Y',Y)=1\) implies \(p>0\).
By the holomorphicity of \(Y\), one obtains
\[
 p_z=\frac12\langle Y',Y \rangle,
 \qquad
 p_{z\bar z}=\frac12|Y'|^2.
\]
The Gram identity then yields
 \begin{align*}
 pp_{z\bar z}-p_zp_{\bar z}
 &=\frac14\left(|Y|^2|Y'|^2-|\langle Y',Y \rangle|^2\right)\nonumber\\
 &=\frac14|\det(Y',Y)|^2
 =\frac14.
 \end{align*}
From Lemma \ref{p-identities} we infer that
\begin{align*}
 u_{z\bar z}
 =-\frac{pp_{z\bar z}-p_zp_{\bar z}}{p^2}
 =-\frac1{4p^2},
 \end{align*}
which is equivalent to \(-\Delta u=p^{-2}=e^{2u}\).

Finally, since \eqref{boundary-p-Y} is equivalent to
\[
 p_r=\lambda p-1\quad\ \mbox{on}\ \mathbb S^1,
\]
we derive
\[
u_r=-\frac{p_r}{p}=\frac1p-\lambda=e^u-\lambda.
\]
The proof is complete.
\end{proof}

From \(\det(Y',Y)=1\), we immediately get \(\det(Y'',Y)=0\).  Thus there exists a holomorphic
function \(q\in A^\infty(\mathbb D)\) such that
\begin{equation}\label{eq-Y-q}
 Y''=qY.
\end{equation}
Combining Lemma \ref{p-identities} with Proposition \ref{normalized-lift} gives
\begin{equation}\label{eq-q-u}
 q=u_z^2-u_{zz}=-\phi.
\end{equation}
We refer to \(q\) as the projective coefficient associated with
the normalized lift \(Y\). Its geometric interpretation in terms of the Schwarzian derivative is
given in the following remark.

\begin{remark}\label{developing-map}
The preceding construction admits a natural interpretation in classical projective geometry. Indeed, the normalized lift $Y=(y_1,y_2)^T$ determines a meromorphic developing map
\[
 f=\frac{y_1}{y_2}:\mathbb D\longrightarrow\widehat\C.
\]
For \(y_2\ne0\), the Wronskian normalization gives
\[
 f'=\frac{y_1'y_2-y_1y_2'}{y_2^2}=\frac1{y_2^2}.
\]
Thus, the normalized lift $Y$ provides a linear representation of the developing map associated with the conformal metric $e^{2u}|dz|^2$.

Since
\[
 e^{-u}=\frac{|Y|^2}{2}=\frac{|y_1|^2+|y_2|^2}{2}
 =\frac{|y_2|^2(1+|f|^2)}{2},
\]
we obtain the classical Liouville representation
\begin{equation*}\label{developing-representation}
 e^u=\frac{2|f'|}{1+|f|^2},
 \qquad
 e^{2u}=\frac{4|f'|^2}{(1+|f|^2)^2}.
\end{equation*}
Moreover, for the Schwarzian derivative
\[
 S(f):=\frac{f'''}{f'}-\frac32\left(\frac{f''}{f'}\right)^2,
\]
a direct calculation from \(Y''=qY\) gives
\begin{equation*}\label{Schwarzian-q}
 S(f)=-2q.
\end{equation*}
Thus \(q\,dz^2\) is precisely the projective differential associated with the developing map.
In particular,
\[
 q\equiv0
 \quad\Longleftrightarrow\quad
 S(f)\equiv0
 \quad\Longleftrightarrow\quad
 f\text{ is a M\"obius transformation}.
\]
This is the classical projective-geometric interpretation of the lift, see
\cite{L1853,H1962,N1949,OS1992}.
\end{remark}

\subsection{The outer factor and Hardy equation}
In this subsection, we construct smooth outer factors and derive the Hardy equation.

We introduce the Euler differential operator
\begin{equation*}\label{D-def}
 D=z\frac{d}{dz},
\end{equation*}
and set
\begin{equation}\label{Phi-def}
 \Phi(z)=z^2q(z).
\end{equation}
Then \eqref{eq-Y-q} is equivalent to
\begin{equation}\label{Euler-Y}
 D(D-1)Y=\Phi Y.
\end{equation}
In particular,
\begin{equation*}\label{Phi-origin}
 \Phi(0)=\Phi'(0)=0.
\end{equation*}

Let
\[
z=e^{i\theta},\quad\ F=|Y|^2\quad\text{on }\mathbb S^1,
\]
and \[\dot{F}=\frac{\partial F}{\partial \theta}.\]

We prove some basic identities on \(\mathbb S^1\), which play a crucial role in the proof of Theorem \ref{Sec-spinor-rigidity}.
\begin{lemma}\label{boundary-projective}
On \(\mathbb S^1\), we have
\begin{equation*}
 \langle DY,Y\rangle=\frac{\lambda F-2-i\dot F}{2},
\end{equation*}
\begin{equation*}\label{Gram-DY}
 F|DY|^2-|\langle DY,Y\rangle|^2=1,
\end{equation*}
\begin{equation*}\label{Phi-boundary-full}
 \Phi F
 =|DY|^2-\frac{\ddot F+\lambda F-2}{2}
 +i\frac{1-\lambda}{2}\dot F,
\end{equation*}
\begin{equation}\label{Im-phi-rep}
 \operatorname{Im}  \Phi=-\frac {a\dot{F}}{2F}.
\end{equation}
\end{lemma}

\begin{proof}
By the holomorphicity of $Y$, we have
\[\partial_rY=DY\quad  \mbox{and}\quad  \dot Y=iDY\ \mbox{ on}\ \mathbb S^1.\]
Combining \eqref{Onofri-E-L} with \eqref{lift-normalization} gives
\[\partial_r|Y|^2=2\langle DY,Y \rangle=\lambda|Y|^2-2\quad\text{on }\ \mathbb S^1.\]
It follows that
\[
 2 \operatorname{Re} \langle DY,Y\rangle=\lambda F-2.
\]
Consequently,
\[
 \dot F=2\operatorname{Re} \langle iDY,Y\rangle=-2 \operatorname{Im} \langle DY,Y\rangle,
\]
and hence
\begin{equation}\label{c-formula}
\langle DY,Y\rangle=\frac{\lambda F-2-i\dot F}{2}.
\end{equation}

By the definition of $D$, we have
\[
 \det(DY,Y)=z\det(Y',Y)=z\quad \mbox{on}\ \mathbb S^1.
\]
Applying the Gram identity then yields
\[
 F|DY|^2-|\langle DY,Y\rangle|^2
 =|\det(DY,Y)|^2=1.
\]

Set
\[
 c=\langle DY,Y\rangle.
\]
It follows from \eqref{Euler-Y} that
\[D^2Y=DY+\Phi Y.\]
Differentiating \(c\) gives
\[
 \begin{aligned}
 \dot c
 &=\langle iD^2Y,Y\rangle+\langle DY,iDY\rangle\\
 &=i\langle D^2Y,Y\rangle-i|DY|^2\\
 &=i\Phi F+ic-i|DY|^2.
 \end{aligned}
\]
Substituting \eqref{c-formula} into this expression yields
\begin{align*}
 \Phi F&=-i\dot c-c+|DY|^2\\
 &=-i\frac{\lambda \dot{F}-i\ddot{F} }{2}-\frac{\lambda F-2-i\dot F}{2}+|DY|^2\\
 &=|DY|^2-\frac{\ddot F+\lambda F-2}{2}
 +i\frac{1-\lambda}{2}\dot F.
 \end{align*}
 Taking imaginary
parts and dividing by \(F>0\), we obtain
\[
 \operatorname{Im}\Phi
 =\frac{1-\lambda}{2}\frac{\dot F}{F}.
\]
\end{proof}

We now construct the scalar outer factor. The idea we shall use is similar to that of \cite{D1970,G2007}.

\begin{lemma}\label{outer-factor}
Let \(F\in C^\infty(\mathbb S^1)\) be a positive function and
\begin{equation}\label{g-def}
 h(z)=\frac1{4\pi}\int_0^{2\pi}
 \frac{e^{it}+z}{e^{it}-z}\log F(e^{it})dt,\quad  g=e^h.
\end{equation}
Then we have
\begin{equation}\label{h-g-propertity}
 h,g\in A^\infty(\mathbb D),
 \qquad
 g(0)>0,
 \qquad
 g\ne0\text{ on }\overline{\mathbb D},
 \qquad
 |g|^2=F\text{ on }\mathbb S^1,
\end{equation}
and \(g\) is uniquely determined by condition \eqref{h-g-propertity}.
\end{lemma}

\begin{proof}
Write
\[
 \log F(e^{it})=\sum_{k\in\mathbb Z}\ell_ke^{ikt}.
\]
Since \(\log F\in C^\infty(\mathbb S^1)\), the Fourier coefficients \(\ell_k\) decay faster than any power.
Using the identity
\[
 \frac{e^{it}+z}{e^{it}-z}
 =1+2\sum_{n=1}^\infty z^ne^{-int},
 \qquad |z|<1
\]
yields
\begin{align*}
 h(z)&=\frac1{4\pi}\int_0^{2\pi}(1+2\sum_{n=1}^\infty z^ne^{-int}) \sum_{k\in\mathbb Z}\ell_ke^{ikt}dt\nonumber\\
 &=\frac{\ell_0}{2}+\sum_{n=1}^\infty\ell_nz^n.
\end{align*}
This implies \(h\in A^\infty(\mathbb D)\) and
\begin{align}\label{Re-h-Fourier}
 \operatorname{Re} h=\frac12\log F\quad \mbox{on}\ \mathbb S^1.
\end{align}
Thus \(g\in A^\infty(\mathbb D)\) has no zeros and \(g(0)=e^{\ell_0/2}>0\),
and \(|g|^2=F\) on \(\mathbb S^1\).

Suppose \(g_1,g_2\) satisfy \eqref{h-g-propertity}. Set
\[\varphi=\frac{g_1}{g_2}.\]
Then \(\varphi\) and \(\frac{1}{\varphi}\) are holomorphic in \(\mathbb D\), continuous on \(\overline{\mathbb D}\), and satisfy
\(|\varphi|=1\) on \(\mathbb S^1\).  Applying the maximum principle to functions \(\varphi\) and \(\frac{1}{\varphi}\) gives
\[|\varphi|=1\quad \mbox{in}\ \mathbb D .\]
The open mapping theorem then implies that
 \(\varphi\) is constant.
Noting that \(g_1(0)>0\) and \(g_2(0)>0\), we derive \(\varphi\equiv1\).
\end{proof}

Applying Lemma \ref{outer-factor} to \(F=|Y|^2\), we obtain the following identities.

\begin{proposition}\label{projective-outer}
The outer factor $h$ satisfies
\begin{equation}\label{Phi-aDh}
 \Phi=aDh
 \quad\text{in }\mathbb D.
\end{equation}
If $a\neq0$ and
\[
 h(z)=\sum_{n=0}^\infty h_nz^n,
 \qquad
 g(z)=\sum_{n=0}^\infty b_nz^n,
\]
then
\begin{equation*}\label{h1-b1-zero}
 h_1=b_1=0.
\end{equation*}
% If $a=0$, then $\Phi=0$.
\end{proposition}

\begin{proof}
Set
\[h=A+iB\quad\ \mbox{on}\ \mathbb S^1.\]
From \eqref{Re-h-Fourier}, we have
\[
 A=\frac12\log F.
\]
On $\mathbb S^1$, the identity \(\dot h=iDh\) gives
\[
 Dh=-i\dot h=\dot B-i\dot A,\qquad
 \operatorname{Im}(aDh)=-a\dot A=-\frac a2\frac{\dot F}{F}.
\]
Comparison with \eqref{Im-phi-rep} yields
\[\operatorname{Im}(\Phi-aDh)=0\quad \mbox{on}\ \mathbb S^1.\]
The maximum principle applied to the harmonic function $\operatorname{Im}(\Phi-aDh)$ shows that it vanishes in $\mathbb D$. The holomorphic function \(\Phi-aDh\) therefore has real range and must be constant.
Observing that \( \Phi(0)=Dh(0)=0\),
we conclude
\[\Phi-aDh=0.\]
This proves \eqref{Phi-aDh}.

A direct calculation gives
\[
 Dh(z)=\sum_{n=1}^\infty nh_nz^n.
\]
From \eqref{Phi-def}, \(\Phi\) has a zero of order at least two at the origin.
Since \(a\neq0\), the coefficient of \(z\) in \(Dh\) vanishes, which implies \(h_1=0\).  Therefore, using \eqref{g-def} gives
\[
 b_1=g'(0)=g(0)h'(0)=b_0h_1=0.
\]
\end{proof}

We now turn to the exact Hardy equation involving $Y$, which arises from the projective structure  $D(D-1)Y=\Phi Y$. To exploit this structure, we define
\begin{equation}\label{eq-X-def}
 X=\frac{Y}{g}.
\end{equation}
Then \(X\in A^\infty(\mathbb D;\mathbb C^2)\) and
\begin{equation}\label{X-inner}
 |X|=1\quad\text{on }\mathbb S^1.
\end{equation}
Thus \(X\) is a smooth \(\mathbb C^2\)-valued function.
The map
$$M_X: f\rightarrow Xf$$
 is an
isometry, i.e.,
\[\| Xf\|_{H^2(\mathbb{D}; \mathbb{C}^2)} = \| f\|_{H^2(\mathbb{D})}, \quad \forall f \in H^2(\mathbb{D}),
\]
see \cite{RR1985,NN2002}.

\begin{proposition}\label{prop-exact-Hardy}
The factors \(Y=gX\) satisfy
\begin{equation}\label{exact-Hardy}
 D(D-1)(gX)=aX Dg.
\end{equation}
Moreover, we have
\begin{equation}\label{Wronskian-X}
 \det(X',X)=g^{-2}.
\end{equation}
\end{proposition}

\begin{proof}
Using \eqref{g-def}, we have  \(Dg=(Dh)g\).
Applying \eqref{Euler-Y} and \eqref{Phi-aDh} then yields
\[
 D(D-1)(gX)=\Phi(gX)=a(Dh)gX=aX Dg.
\]
Moreover,
\[
 X'=\frac{Y'}g-\frac{g'}{g^2}Y.
\]
The identity \(\det(Y,Y)=0\) gives
\[
 \det(X',X)
 =g^{-2}\det(Y',Y)
 =g^{-2}.
\]
\end{proof}

\subsection{Hardy spectral majorization and moment identities}
\

In this subsection, we establish a Hardy spectral majorization principle and derive the corresponding moment identities, which play a central role in the endpoint rigidity argument, namely $\lambda=2$.

For a scalar or vector-valued Hardy function
\[
 f(z)=\sum_{n=0}^\infty f_nz^n,
\]
we use the normalized norm
\[
 \| f\|_{H^2}^2
 =\frac1{2\pi}\int_0^{2\pi}|f(e^{i\theta})|^2\,d\theta
 =\sum_{n=0}^\infty|f_n|^2,
\]
where the last equality follows from Parseval's identity.
Let \( H^\infty(\mathbb D;\mathbb C^m)\) denote the space of bounded holomorphic functions under the norm
$$\|f\|_{H^\infty}=\sup_{|z|<1} |f(z)|= \operatorname*{ess\,sup}_{0 \le \theta < 2\pi} |f(e^{i\theta})|$$
Let \(P_N\) denote the orthogonal projection onto polynomials of degree at most \(N\).

We now establish spectral majorization for the Taylor coefficients of $Xf$ and $f$, which follows from the inner multiplier $X$.

\begin{lemma}\label{spectral-majorization}
Let \(X\in H^\infty(\mathbb D;\mathbb C^m)\) satisfy \(|X|=1\) almost everywhere on \(\mathbb S^1\), and
\[
 f(z)=\sum_{n=0}^\infty b_nz^n\in H^2(\mathbb D),
 \qquad
 Xf=\sum_{n=0}^\infty A_nz^n.
\]
Then, for every \(N\ge0\),
\begin{equation*}\label{less-partial-majorization}
 \sum_{n=0}^N|A_n|^2\le\sum_{n=0}^N|b_n|^2.
\end{equation*}
Furthermore, if \((w_n)_{n\ge0}\) is a nonnegative, nondecreasing sequence and
$$\sum_{n=0}^\infty w_n|A_n|^2<\infty,\quad \sum_{n=0}^\infty w_n|b_n|^2<\infty,$$
then
\begin{equation}\label{weighted-majorization}
 \sum_{n=0}^\infty w_n|A_n|^2
 \ge
 \sum_{n=0}^\infty w_n|b_n|^2.
\end{equation}
If equality holds in \eqref{weighted-majorization}, then
\begin{equation}\label{equality-partial}
 \sum_{n=0}^N|A_n|^2=\sum_{n=0}^N|b_n|^2
\end{equation}
for every \(N\) with \(w_{N+1}>w_N\).
\end{lemma}

\begin{proof}
Define \(M_Xf=Xf\).  Since \(|X|=1\) almost everywhere on \(\mathbb S^1\), \(M_X\) is an isometry, i.e.,
\[
 \| M_Xf\|_{H^2}=\| f\|_{H^2}.
\]
Since \(X\) has only nonnegative Fourier modes, the first \(N+1\) Taylor coefficients of \(Xf\) depend only on the coefficients of \(f\) up to degree \(N\).  This gives the operator identity
\[
 P_NM_X=P_NM_XP_N.
\]
A straightforward computation shows that
\[
 \begin{aligned}
 \sum_{n=0}^N|A_n|^2&=\| P_N(Xf)\|_{H^2}\\
 &=\| P_NM_XP_Nf\|_{H^2}\\
 &\le\| M_XP_Nf\|_{H^2}\\
 &=\| P_Nf\|_{H^2}=\sum_{n=0}^N|b_n|^2.
 \end{aligned}
\]

Set
\[
 \delta_N=\sum_{n=0}^N\bigl(|b_n|^2-|A_n|^2\bigr)\ge0.
\]
The isometry \(M_X\) yields
\begin{equation}\label{equality-sum-unweight}
 \| M_Xf\|_{H^2}=\| Xf\|_{H^2}=\| f\|_{H^2}= \sum_{n=0}^{\infty}|A_n|^2=\sum_{n=0}^{\infty}|b_n|^2
\end{equation}
and \(\delta_N\to0\) as $N\rightarrow\infty$.  For each
\(M\), summation by parts yields
\begin{equation}\label{Abel-finite}
 \sum_{n=0}^M w_n\bigl(|A_n|^2-|b_n|^2\bigr)
 =\sum_{N=0}^{M-1}(w_{N+1}-w_N)\delta_N-w_M\delta_M.
\end{equation}
Moreover, from \eqref{equality-sum-unweight} we obtain
\[
 \delta_M=\sum_{n>M}\bigl(|A_n|^2-|b_n|^2\bigr).
\]
Hence
\begin{equation}\label{partial-sum}
 0\le w_M\delta_M
 \le w_M\sum_{n>M}|A_n|^2
 \le\sum_{n>M}w_n|A_n|^2
 \longrightarrow0.
\end{equation}
Passing to the limit in \eqref{Abel-finite}, we derive
\begin{equation}\label{total-Abel-infinite}
 \sum_{n=0}^\infty w_n\bigl(|A_n|^2-|b_n|^2\bigr)
 =\sum_{N=0}^\infty(w_{N+1}-w_N)\delta_N\ge0.
\end{equation}
Using \eqref{Abel-finite}-\eqref{total-Abel-infinite}, one obtains \eqref{weighted-majorization}.

If equality holds in \eqref{weighted-majorization}, then \eqref{total-Abel-infinite} becomes
\[
\sum_{N=0}^\infty(w_{N+1}-w_N)\delta_N=0.
\]
All terms in this series are nonnegative. For every index satisfying \(w_{N+1}>w_N\), the corresponding term vanish only when
\(\delta_N=0\).
This proves \eqref{equality-partial}.
\end{proof}

Next we establish the following first and second-moments identities.
 Set
\begin{equation}\label{coefficients-all}
 Y(z)=\sum_{n=0}^\infty A_nz^n,
 \qquad
 g(z)=\sum_{n=0}^\infty b_nz^n,
 \qquad
 X(z)=\sum_{n=0}^\infty X_nz^n.
\end{equation}

\begin{proposition}\label{prop-moment-identities}
Under the hypothesis of Proposition \ref{prop-exact-Hardy}, we have
\begin{equation}\label{first-moment}
 \sum_{n=0}^\infty n(n-1)|A_n|^2
 =a\sum_{n=0}^\infty n|b_n|^2,
\end{equation}
\begin{equation}\label{second-moment}
 \sum_{n=0}^\infty n^2(n-1)^2|A_n|^2
 =a^2\sum_{n=0}^\infty n^2|b_n|^2.
\end{equation}
\end{proposition}

\begin{proof}
Applying Proposition \ref{prop-exact-Hardy} and \eqref{coefficients-all}, we have
\[
 \begin{aligned}
 \sum_{n=0}^\infty n(n-1)|A_n|^2
 &=\langle D(D-1)Y,Y\rangle_{H^2}\\
 &=a\langle XDg,Xg\rangle_{H^2}\\
 &=a\langle Dg,g\rangle_{H^2}\\
 &=a\sum_{n=0}^\infty n|b_n|^2,
 \end{aligned}
\]
and
\begin{equation}\label{com-second-moment}
 \parallel D(D-1)Y\parallel_{H^2}^2
 =a^2\parallel XDg\parallel_{H^2}^2
 =a^2\parallel Dg\parallel_{H^2}^2.
\end{equation}
From \eqref{com-second-moment}, we immediately derive \eqref{second-moment}.
\end{proof}

\subsection{Proof of Theorem \ref{radial-rigidity} }
\
In this subsection, we complete the proof of Theorem \ref{radial-rigidity}.
To this end, we establish the following holomorphic rigidity result.

\begin{theorem}\label{Sec-spinor-rigidity}
Let $0<\lambda\leq2$, and let $ Y\in A^{\infty}(\mathbb D;\mathbb C^2)$
satisfy
\begin{equation}\label{abstract-spinor-system}
 \det(Y',Y)=1\quad\text{in }\mathbb D,
 \qquad
 \partial_r|Y|^2=\lambda|Y|^2-2\quad\text{on }\mathbb S^1.
\end{equation}
Then \(Y\) is affine, i.e.,
\[
 Y(z)=A_0+A_1z
\]
for some \(A_0,A_1\in\mathbb C^2\).
\end{theorem}

\begin{proof}
From Proposition \ref{projective-outer}, it follows that \(b_1=0\).  Applying Lemma
\ref{spectral-majorization} with
\[
 w_n=n(n-1)
\]
and using \eqref{first-moment}, we obtain
\begin{equation}\label{key-first-chain}
 a\sum_{n=0}^\infty n|b_n|^2
 =\sum_{n=0}^\infty n(n-1)|A_n|^2
 \ge\sum_{n=0}^\infty n(n-1)|b_n|^2
 \ge \sum_{n=0}^\infty n|b_n|^2\geq 0.
\end{equation}
Here the last inequality in \eqref{key-first-chain} follows from \(b_1=0\) and
\[
 n(n-1)\ge n,\quad \mbox{for}\ n\ge2.
\]

For $0<\lambda<1$, one has \(a<0\).
Substituting this into \eqref{key-first-chain} yields
\[\sum_{n=0}^\infty n|b_n|^2=0.\]
Similarly, for $1<\lambda<2$, the same argument gives \(\sum_{n=0}^\infty n|b_n|^2=0\).
Consequently, by \eqref{coefficients-all}, the function \(g\) is
constant.  Combining \eqref{g-def} with \eqref{Phi-aDh}, we conclude that
\begin{equation*}\label{big-range-zero}
\Phi\equiv0.
\end{equation*}
For $\lambda=1$, we have $a=0$. Applying \eqref{Phi-aDh} then yields $\Phi\equiv0$, and hence
 \eqref{Euler-Y} implies
\[
 D(D-1)Y=0.
\]
From the expansion \(Y=\sum_{n=0}^\infty A_nz^n\), it follows that \(n(n-1)A_n=0\) for every \(n\), which gives
\begin{equation}\label{key-afine-equ}
 Y=A_0+A_1z.
\end{equation}

Now we consider the case \(a=1\).  From \eqref{key-first-chain} we infer that
$$ a\sum_{n=0}^\infty n|b_n|^2
 =\sum_{n=0}^\infty n(n-1)|A_n|^2
 =\sum_{n=0}^\infty n(n-1)|b_n|^2$$
and equality holds in \eqref{weighted-majorization} for \(w_n=n(n-1)\).
Thus \eqref{equality-partial} holds.
Since
\[
 w_0=w_1=0,
 \qquad
 w_2=2>w_1,
\]
applying \eqref{equality-partial} gives
\begin{equation}\label{endpoint-low-modes}
 |A_0|^2+|A_1|^2=|b_0|^2+|b_1|^2=|b_0|^2.
\end{equation}
Recalling that \(Y=Xg\) and \(b_1=0\), we derive
\[
 A_0=b_0X_0,
 \qquad
 A_1=b_0X_1.
\]
As \(b_0=g(0)>0\), combining this with \eqref{endpoint-low-modes} yields
\begin{equation}\label{X01-unit}
 |X_0|^2+|X_1|^2=1.
\end{equation}
On the other hand, applying Parseval's identity and \(|X|=1\) on \(\mathbb S^1\) gives
\begin{equation}\label{sum-X01-unit}
 \sum_{n=0}^\infty|X_n|^2=1.
\end{equation}
Combining \eqref{X01-unit} with \eqref{sum-X01-unit}, we conclude that \(X_n=0\) for all \(n\ge2\), and hence
\[
 X=X_0+X_1z.
\]
The Wronskian identity \eqref{Wronskian-X} then reads
\[
 g^{-2}=\det(X',X)=\det(X_1,X_0).
\]
Since \(g\) has no zeros, it follows that \(g\) is constant.  Therefore,
\begin{equation*}\label{equiv-P-zero}
\Phi\equiv0.
\end{equation*}
Arguing as \eqref{key-afine-equ}, \(Y\) is affine.
This completes the proof.
\end{proof}

Using Theorem \ref{Sec-spinor-rigidity}, we now proceed to the proof of Theorem \ref{radial-rigidity}.

\textbf{Proof of Theorem \ref{radial-rigidity}.}
We divide the proof into the following  two  cases.

\textbf{\textbf{Case $1.$}} $\lambda\leq0$.
Using equation \eqref{Onofri-E-L} and integration by parts, we have
\begin{align*}
\int_{\mathbb{B}^2} e^{2u}dx&=-\int_{\mathbb{B}^2}\Delta udx=-\int_{\mathbb{S}^{1}}\frac{\partial u}{\partial\nu}d\sigma
=-\int_{\mathbb{S}^{1}}(e^u-\lambda)d\sigma.
\end{align*}
Hence
\begin{align*}
0<\int_{\mathbb{B}^2} e^{2u}dx+\int_{\mathbb{S}^{1}}e^ud\sigma=2\pi\lambda\leq0.
\end{align*}
This is a contradiction. Therefore, $u$ does not exist.

\textbf{\textbf{Case $2.$}} $0<\lambda\leq2$, $\lambda\neq1$.
Let \(Y\) be the normalized lift from Proposition \ref{normalized-lift}.  From
Theorem \ref{Sec-spinor-rigidity}, we have
\[
 Y(z)=A_0+A_1z.
\]
The Wronskian normalization gives
\begin{equation*}\label{eq-affine-wronskian}
 \det(Y',Y)=\det(A_1,A_0)=1.
\end{equation*}

Since \(z=e^{i\theta}\) on \(\mathbb S^1\), one has
\begin{equation}\label{com-affine}
 |A_0+rzA_1|^2
 =|A_0|^2+r^2|A_1|^2
 +2r\operatorname{Re} \bigl(z\langle A_1,A_0\rangle\bigr).
\end{equation}
Differentiating \eqref{com-affine} with respect to $r$ at $r=1$ yields
\begin{equation}\label{direct-affine}
 \left.\partial_r|A_0+rzA_1|^2\right|_{r=1}
 =2|A_1|^2+2\operatorname{Re} \bigl(z\langle A_1,A_0\rangle\bigr).
\end{equation}
On the other hand, the boundary identity \eqref{abstract-spinor-system} reads
\begin{equation}\label{affine-differ}
 \left.\partial_r|A_0+rzA_1|^2\right|_{r=1}
 =\lambda|A_0+zA_1|^2-2.
\end{equation}
Comparing its first Fourier modes with \eqref{com-affine}-\eqref{affine-differ} gives
\[
 (1-\lambda)\langle A_1,A_0\rangle=0.
\]
Since \(0<\lambda\leq 2\) with $\lambda\neq1$, it follows that
\begin{equation}\label{eq-A-orthogonal}
 \langle A_1,A_0\rangle=0.
\end{equation}
Together with the Gram identity, we have
\[
 1=|\det(A_1,A_0)|^2
 =|A_1|^2|A_0|^2-|\langle A_1,A_0\rangle|^2
 =|A_1|^2|A_0|^2.
\]
Set
\[
 R=|A_1|^2>0.
\]
Then \(|A_0|^2=R^{-1}\), and consequently
\[
 e^{-u(z)}=\frac{|Y(z)|^2}{2}
 =\frac{R^{-1}+R|z|^2}{2}.
\]
Therefore, the solution takes the explicit form
$$u(z)=\log\frac{2R}{1+R^2|z|^2},$$

Comparing the constant Fourier coefficients in \eqref{affine-differ} and using \eqref{com-affine}-\eqref{eq-A-orthogonal} gives
\[
 2R=\lambda\left(R+\frac1R\right)-2.
\]
Consequently, $R$ satisfies
\begin{equation}\label{quar-equa-R}
(\lambda-2)R^{2}-2R +\lambda=0.
\end{equation}
If \(0<\lambda<2\), the polynomial
\[
 P_\lambda(R)=(\lambda-2)R^2-2R+\lambda
\]
is strictly decreasing on \((0,\infty)\). Since \(P_\lambda(0)=\lambda>0\) and
\[\lim_{R\rightarrow\infty}P_\lambda(R)=-\infty,\]
$P_\lambda(R)$ admits a unique positive root.  If \(\lambda=2\), the
equation \eqref{quar-equa-R} reduces to \(R=1\).

A direct calculation gives
\[
 -\Delta\left(\log\frac{2R}{1+R^2|z|^2}\right)
 =\frac{4R^2}{(1+R^2|z|^2)^2}=e^{2u}\quad\mbox{in}\ \mathbb D.
\]
and
\[
 \partial_ru+\lambda-e^u
 =\frac{(\lambda-2)R^2-2R+\lambda}{1+R^2}=0\quad\mbox{on}\ \mathbb S^1.
\]
This completes the proof of Theorem \ref{radial-rigidity}.
\hfill\qedsymbol

\section{Classification of solutions for $\lambda=1$}

In this section,  we finish the proof of Theorem \ref{critical-classification}. We first construct a Hardy-spinor global developing map, and then analyze its projective differential, which identifies the underlying M\"{o}bius structure at $\lambda=1$ and leads to the classification.

\textbf{\textbf{ Proof of Theorem \ref{critical-classification}. }}
We divide the proof into four steps.

\textbf{Step 1. Construction of the global developing map.}
From Theorem \ref{Sec-spinor-rigidity}, we know that \(Y\) is affine, i.e.,
\[
 Y(z)=\binom{y_1(z)}{y_2(z)}=A_0+A_1z=\binom{a+bz}{c+dz}
\]
for some \(a,b,c,d\in\mathbb C\).
By the Wronskian normalization, we have
\[
1=\det(Y',Y)
=y_1'y_2-y_1y_2'
=bc-ad.
\]
Consequently,
\[
f(z)=\frac{y_1(z)}{y_2(z)}
\]
is the restriction to \(\mathbb D\) of a M\"{o}bius automorphism of the
Riemann sphere \(\widehat{\mathbb C}\), namely
$$f=\frac{a+bz}{c+dz},\quad a,b,c,d\in\mathbb C,\ \ bc-ad\neq0.$$
In particular, if
\(y_2(z_0)=0\), then \(f(z_0)=\infty\). Such a point is not a singularity of the developing map when \(f\) is to be interpreted as a sphere-valued holomorphic map.

For \(y_2\ne0\), the Wronskian normalization yields
\[
 f'=\frac{y_1'y_2-y_1y_2'}{y_2^2}=\frac1{y_2^2}.
\]
Using
\[
 e^{-u}=\frac{|Y|^2}{2}=\frac{|y_1|^2+|y_2|^2}{2}
 =\frac{|y_2|^2(1+|f|^2)}{2},
\]
we derive the standard Liouville form
\begin{equation*}\label{lamda-developing-representation}
 e^u=\frac{2|f'|}{1+|f|^2},
 \qquad
 e^{2u}=\frac{4|f'|^2}{(1+|f|^2)^2}.
\end{equation*}
When \(y_2=0\), we introduce the spherical coordinate
\[
\zeta=\frac1f=\frac{y_2}{y_1}.
\]
The identity \(\zeta'=-\frac1{y_1^2}\)
 gives
\[
\frac{4|\zeta'|^2}{(1+|\zeta|^2)^2}
=
\frac4{\left(|y_1|^2+|y_2|^2\right)^2}
=
e^{2u}.
\]
Patching together the chartwise identities yields the global identity
\[
e^{2u}|dz|^2
=
f^*g_{\mathbb S^2},
\]
where $f^*g_{\mathbb S^2}$ is the pullback metric on $\mathbb D$ and
\[
g_{\mathbb S^2}
=
\frac{4|dw|^2}{(1+|w|^2)^2}
\]
is the round metric of constant curvature one.

Since \(f\) is the restriction of a global M\"{o}bius automorphism, it is
injective and maps \(\mathbb D\) biholomorphically onto the open subset
\[
\Omega=f(\mathbb D)\subset\mathbb S^2.
\]

\textbf{\textbf{Step 2. Determination of the spherical image.}}
Since a M\"{o}bius automorphism maps generalized circles to generalized circles, the image
\[
\Gamma=f(\mathbb S^1)
\]
is a spherical circle, and \(\Omega\) is one of the two connected
components of \(\mathbb S^2\setminus \Gamma\). In particular, there exist
\(c\in\mathbb S^2\) and \(\rho\in(0,\pi)\) such that
\begin{equation}\label{Omega-Ball}
\Omega=B_\rho^{\mathbb S^2}(c).
\end{equation}

We now compute the geodesic radius \(\rho\). We first fix the sign convention for the boundary curvature.
For a Riemannian metric \(g\), we define the signed geodesic curvature
with respect to the inward unit normal \(n\) by
\[
k_g=\langle\nabla_TT,n\rangle_g,
\]
where \(T\) is unit tangent vector along the boundary of $\Omega$.

Define
\[
g_0=|dz|^2,
\qquad
g=e^{2u}g_0.
\]
On the boundary \(\mathbb S^1\), let \(T_0\) be a Euclidean
unit tangent and define
the Euclidean inward unit normal by \[
n_0=-\nu.
\]
Then we have
\[
\nabla^{g_0}_{T_0}T_0=n_0.
\]
With respect to the metric $g$, the corresponding unit tangent vector $T$  and inward unit normal $n$ are given by
\[
T=e^{-u}T_0,
\qquad
n=e^{-u}n_0.
\]
The conformal change formula for the connection
\[
\nabla^g_XY
=
\nabla^{g_0}_XY
+X(u)Y+Y(u)X
-g_0(X,Y)\nabla^{g_0}u
\]
gives
\[
\nabla^g_TT
=
e^{-2u}
\left(
\nabla^{g_0}_{T_0}T_0
+T_0(u)T_0
-\nabla^{g_0}u
\right).
\]
Thus, the geodesic curvature is
\[
k_g=g(\nabla^g_TT,n)=e^{-u}\left(1-\partial_{n_0}u\right)=
e^{-u}\left(1+\partial_\nu u\right).
\]
Invoking  \eqref{Onofri-E-L} yields
\[
k_g=1
\qquad\text{on }S^1.
\]

Since
\[
f:(\mathbb D,g)\longrightarrow(\Omega,g_{\mathbb S^2})
\]
is an isometry and maps the inward side of \(\mathbb S^1\) onto the inward
side of \(\partial\Omega\), the geodesic curvature of
\(\partial\Omega\), computed with respect to the inward normal of
\(\Omega\), is also equal to one.

In geodesic polar coordinates centered at \(c\), the round metric is
\[
g_{\mathbb S^2}
=
dr^2+\sin^2r\,d\theta^2.
\]
Along the boundary \(r=\rho\), the unit tangent and the inward unit normal of
\(B_\rho^{\mathbb S^2}(c)\) are
\[
T=\frac1{\sin\rho}\partial_\theta,
\qquad
n=-\partial_r.
\]
Direct computation of the connection yields
\[
\nabla_TT=-\cot\rho\,\partial_r,
\]
and hence
\[
k_{\partial B_\rho(c)}
=
\langle\nabla_TT,n\rangle
=
\cot\rho.
\]
The isometry condition therefore forces
\begin{equation}\label{cot-equa}
\cot\rho=1.
\end{equation}
Since \(\rho\in(0,\pi)\), the equation \eqref{cot-equa} has the unique solution
\begin{equation}\label{cot-value}
\rho=\frac{\pi}{4}.
\end{equation}
In particular, the sign of the boundary condition of \eqref{Onofri-E-L} selects the smaller
spherical cap, rather than its complementary cap.

Choose an orientation-preserving spherical isometry
\[
M\in PSU(2)
\]
such that \(M(c)=0\), where \(0\) denotes the origin in the
stereographic coordinate \(w\). Set
\[
F=M\circ f.
\]
Combining \eqref{Omega-Ball} with \eqref{cot-value} gives
\[
F(\mathbb D)=B_{\pi/4}^{\mathbb S^2}(0).
\]
The spherical distance from \(0\)
to \(w\) with respect to \(g_{\mathbb S^2}\) is
\[
d_{\mathbb S^2}(0,w)
=
\int_0^{|w|}\frac{2\,dt}{1+t^2}
=
2\arctan|w|.
\]
Thus
\[
B_{\pi/4}^{\mathbb S^2}(0)
=
\{w\in\mathbb C:|w|<R\},
\qquad
R=\tan\frac{\pi}{8}=\sqrt2-1.
\]
Hence
\[
F:\mathbb D\longrightarrow\mathbb D_R
\]
 is a biholomorphism,
with $\mathbb D_R=\{w\in\mathbb C:|w|<R\}$, and consequently
\[
\psi=\frac{F}{R}\in \operatorname{Aut}(\mathbb D).
\]

Since \(M\) preserves the round metric, we have
\[
e^{2u}|dz|^2
=
F^*g_{\mathbb S^2}
=
\frac{4|F'|^2}{(1+|F|^2)^2}|dz|^2.
\]
Substituting \(F=R\psi\) yields
\[
e^u
=
\frac{2|F'|}{1+|F|^2}
=
\frac{2R|\psi'|}
{1+R^2|\psi|^2}.
\]
Therefore, the solution takes the form
\begin{equation}\label{R-critical-sol}
u(z)=\log\frac{2R|\psi'(z)|}{1+R^2|\psi(z)|^2}.
\end{equation}

\textbf{\textbf{Step 3. The function \eqref{R-critical-sol} satisfies \eqref{Onofri-E-L}.}}
Define
\[
u_R(w)
=
\log\frac{2R}{1+R^2|w|^2}.
\]
For \(R=\sqrt2-1\), a direct computation gives
\[
-\Delta u_R
=\frac{4R^2}{(1+R^2|w|^2)^2}
=e^{2u_R}\qquad\text{in }\mathbb D,
\]
\[
\partial_\nu u_R+1
=-\frac{2R^2}{1+R^2}+1
=\frac{1-R^2}{1+R^2}=\frac{2R}{1+R^2}=
e^{u_R}\qquad\text{on } \mathbb S^1.
\]

Let \(\psi\in\operatorname{Aut}(\mathbb D)\) and set
\[
u
=
u_R\circ\psi+\log|\psi'|.
\]
Every automorphism of unit disk can be written as
\[
\psi(z)
=
e^{i\beta}\frac{z-a}{1-\overline a z},
\qquad
a\in\mathbb D,\ \ \beta\in\mathbb R.
\]
In particular, \(\psi'\) is holomorphic and non-vanishing in a
neighborhood of \(\overline{\mathbb D}\). Hence
\(\log|\psi'|\) is harmonic. By conformal invariance of the Laplacian in two dimensions, we derive
\[
-\Delta u
=-|\psi'|^2(\Delta u_R)\circ\psi=
|\psi'|^2e^{2u_R\circ\psi}
=e^{2u}
\qquad\text{in }\mathbb D.
\]

It remains to verify the boundary condition of \eqref{Onofri-E-L}.
The derivative of $\psi$ is
\[
\psi'(z)
=
e^{i\beta}
\frac{1-|a|^2}{(1-\overline a z)^2}.
\]
On \(\mathbb S^1\), using
\[
\bar{z}=\frac{1}{z}\quad\mbox{and}\quad |\psi(z)|=1,
\]
we obtain
$$|z-a|^2=1+|a|^2-(a\bar{z}+\bar{a}z)=|1-\bar{a}z|^2=1+|a|^2-(\frac{a}{z}+\bar{a}z).$$
Consequently,
\begin{equation*}\label{preserve-boundary-p}
\frac{z\psi'(z)}{\psi(z)}=
\frac{z(1-|a|^2)}{(1-\bar{a}z)(z-a)}
=\frac{z(1-|a|^2)}{z(1+|a|^2)-(a+\bar{a}z^2)}
=
\frac{1-|a|^2}{|z-a|^2}
=
|\psi'(z)|,
\end{equation*}
which implies
\[
D\psi(z)[\nu_z]=\frac{z}{|z|}\psi'(z)=|\psi'(z)|\psi(z)=|\psi'(z)|\frac{\psi(z)}{|\psi(z)|}
=
|\psi'(z)|\,\nu_{\psi(z)}.
\]
Hence, for every \(h\in C^1(\overline{\mathbb D})\),
\begin{equation}\label{h-differ-comp}
\partial_\nu(h\circ\psi)(z)
=
|\psi'(z)|
(\partial_\nu h)(\psi(z)).
\end{equation}

Moreover, since \(|\bar{a} z|=|a|\) for \(z\in\mathbb S^1\), we have
\begin{align}\label{log-differ}
\partial_\nu\log|\psi'(z)|
&=\left.
\frac{d}{dr}
\right|_{r=1}
\log|\psi'(rz)|\nonumber\\
&=\operatorname{Re}\frac{d}{dr}
|_{r=1}
\log(\psi'(rz))\nonumber\\
&=\operatorname{Re} \frac{z\psi''(z)}{\psi'(z)}\nonumber\\
&=2\operatorname{Re}\frac{\bar{a} z}{1-\bar{a} z}\nonumber\\
&=\frac{2(\operatorname{Re}(\bar{a} z)-|\bar{a} z|^2)}
{|1-\bar{a} z|^2}\nonumber\\
&=|\psi'(z)|-1.
\end{align}
Combining \eqref{h-differ-comp} with \eqref{log-differ} gives
\[
\begin{aligned}
\partial_\nu u+1
&=
|\psi'|(\partial_\nu u_R)\circ\psi
+\partial_\nu\log|\psi'|+1\\
&=
|\psi'|(\partial_\nu u_R)\circ\psi
+|\psi'|\\
&=
|\psi'|
\bigl(\partial_\nu u_R+1\bigr)\circ\psi\\
&=
|\psi'|e^{u_R\circ\psi}
=
e^u.
\end{aligned}
\]
Thus every function \eqref{R-critical-sol} is a solution of \eqref{Onofri-E-L}.

\textbf{\textbf{Step 4. Uniqueness modulo rotations.}}
Suppose that \(\psi_1,\psi_2\in\operatorname{Aut}(\mathbb D)\)
give the same function \(u\). Set
\[
F_j=R\psi_j,
\qquad j=1,2.
\]
Each \(F_j\) maps \(\mathbb D\) biholomorphically onto
\(\mathbb D_R\), extends to a M\"{o}bius automorphism of the Riemann sphere
\(\widehat{\mathbb C}\), and satisfies
\[
F_j^*g_{\mathbb S^2}
=
e^{2u}|dz|^2.
\]
Define
\[
G=F_2\circ F_1^{-1}.
\]
Then \(G\) is the restriction of a M\"{o}bius automorphism of
\(\widehat{\mathbb C}\), maps \(\mathbb D_R\) onto \(\mathbb D_R\), and
\[
G^*g_{\mathbb S^2}
=
g_{\mathbb S^2}
\qquad\text{on }\mathbb D_R.
\]

We now verify that \(G\in PSU(2)\). Choose
\[
A=
\begin{pmatrix}
a&b\\
c&d
\end{pmatrix}
\in SL(2,\mathbb C)
\]
such that
\begin{equation}\label{G-Mobius-exp}
G(w)=\frac{aw+b}{cw+d}.
\end{equation}
Since \(G\) preserves \(g_{\mathbb S^2}\) on \(\mathbb D_R\), we obtain
\begin{equation*}\label{G-metric-sphere}
\frac{|G'(w)|^2}{(1+|G(w)|^2)^2}
=
\frac1{(1+|w|^2)^2}.
\end{equation*}
Together with \(\det A=1\), this yields
\[
|aw+b|^2+|cw+d|^2
=
1+|w|^2
\qquad\text{on }\mathbb D_R.
\]
Then we have
\begin{equation}\label{comp-cofficient}
|a|^2+|c|^2=1,
\qquad
|b|^2+|d|^2=1,
\qquad
a\overline b+c\overline d=0.
\end{equation}
Equivalently,
\[
A^*A=I.
\]
Hence \(A\in SU(2)\), and therefore \(G\in PSU(2)\).

%Note that
%\[
%\mathbb D_R
%=
%B_{\pi/4}^{\mathbb S^2}(0).
%\]
Since \(G\) is a spherical isometry and
\(G(\mathbb D_R)=\mathbb D_R\), we have
\[
B_{\pi/4}^{\mathbb S^2}(G(0))
=
B_{\pi/4}^{\mathbb S^2}(0).
\]
Since $\rho=\pi/4\neq \pi/2$, the spherical ball $B_{\pi/4}^{\mathbb S^2}(0)$ has a unique
center, which forces
\[
G(0)=0.
\]
Substitution this into \eqref{G-Mobius-exp} and using \eqref{comp-cofficient} yields \(b=0\) and \( c=0\),
so
\[
G(w)=e^{i\theta}w
\]
for some \(\theta\in\mathbb R\). Consequently,
\[
R\psi_2
=
G\circ(R\psi_1)
=
e^{i\theta}R\psi_1,
\]
and hence
\(
\psi_2=e^{i\theta}\psi_1.
\)
Thus \(\psi\) is unique up to the rotational factor $e^{i\theta}$ for some $\theta\in \mathbb R$.

\section{The proof of Theorem \ref{pde-quadratic} }
\

This section is devoted to the proof of  Theorem \ref{pde-quadratic}. More precisely, we derive a dichotomy: solutions are either radial, or they must have a nonzero quadratic obstruction at the origin for $2<\lambda\leq3$.

\subsection{Quadratic obstruction for $\lambda>2$}
\

In this subsection, we establish a quantitative bound for the higher-order modes \(n\geq3\) using the second moment.

From Section 2, we know that for \(a>1\) (equivalently \(\lambda>2\)), the first moment no longer forces the outer factor $g$ to be constant.
However, the first moment selects the only possible leading nonconstant outer mode when \(a<2\).  The second
moment then controls the full higher frequency tail. This leads to the following theorem, which plays a key role in proving Theorem \ref{pde-quadratic}.

\begin{theorem}\label{quadratic-obstruction}
Let \(1<a\leq2\) and \(Y\) satisfy
\eqref{abstract-spinor-system}.  Let \(g\) be the unique normalized scalar outer function satisfying
\[
 g(0)>0,
 \qquad
 |g|^2=|Y|^2\quad\text{on }\mathbb S^1,
\]
Then \(b_1=0\), and exactly one of the following holds.

 $(i)$ \(g\) is constant, and hence \(Y\) is affine;

 $(ii)$ \(b_2\ne0\) and
\begin{equation}\label{tail-general}
 \sum_{n=3}^\infty n^2\bigl((n-1)^2-a^2\bigr)|b_n|^2
 \le 4(a^2-1)|b_2|^2.
\end{equation}
\end{theorem}

\begin{proof}
From Proposition \ref{projective-outer}, we obtain \(b_1=0\).
Applying Lemma \ref{spectral-majorization} with
\(w_n=n^2(n-1)^2\) and using \eqref{second-moment} yield
\begin{equation}\label{second-outer-inequality}
 \sum_{n=2}^\infty n^2\bigl((n-1)^2-a^2\bigr)|b_n|^2\le0.
\end{equation}
For \(n=2\), we have
\[
 n^2\bigl((n-1)^2-a^2\bigr)|b_n|^2= 4(1-a^2)|b_2|^2
\]
and \eqref{tail-general} follows immediately from \eqref{second-outer-inequality}.

It remains to establish the dichotomy. Suppose that \(g\) is constant. Combining \eqref{g-def} with \eqref{Phi-aDh} gives
\[
\Phi=aDh=0.
\]
Arguing as \eqref{key-afine-equ},  we derive that alternative \({\rm(i)}\) holds.

Suppose now that \(g\) is nonconstant. We claim that
\[
b_2\neq0.
\]
To prove this,  assume to the contrary that \(b_2=0\).
If \(1<a<2\), then for \(n\geq3\), we have
\[
(n-1)^2-a^2
\geq4-a^2>0.
\]
Substitution this estimate into \eqref{tail-general} yields
\[\sum_{n=3}^\infty n^2(4-a^2\bigr)|b_n|^2\leq\sum_{n=3}^\infty n^2\bigl((n-1)^2-a^2\bigr)|b_n|^2=0\]
It follows that
\[
b_n=0
\qquad\text{for every }n\geq3.
\]
Together with \(b_1=b_2=0\), this gives
\[
g\equiv b_0,
\]
which contradicts the assumption that \(g\) is nonconstant. Thus \(b_2\neq0.\)

For \(a=2\), the spectral coefficient $(n-1)^2-4$ vanishes at $n=3$, whereas
\[
(n-1)^2-4=(n-3)(n+1)>0
\qquad\text{for every }n\geq4.
\]
Inequality \eqref{tail-general} gives \(b_n=0\) for every $n\geq4$.
Consequently
\begin{equation}
\label{endpoint-cubic-g}
g(z)=b_0+b_3z^3.
\end{equation}

We show that \(b_3=0\).
Using \(D(D-1)f=z^2f''\) and \(Dg=zg'\)
in \eqref{exact-Hardy}, we obtain
\[
z^2(gX)''=2zg'X.
\]
Together with \eqref{endpoint-cubic-g}, this implies \(z^2g''=2zg'\),
and thus
\[
z^2\bigl((gX)''-g''X\bigr)=0.
\]
The preceding identity gives
\[
(gX)''=g''X
\qquad\text{in }\mathbb D\setminus \{0\},
\]
and holomorphicity of $X$ and $g$ extends the identity across $z=0$. This implies
\[
g(2g'X'+gX'')=(g^2X')'=0.
\]
Hence there exists a constant vector \(V\in\mathbb C^2\) such that
\[
g^2X'=V.
\]

Define
\[
H(z)=\int_0^z\frac{d\zeta}{g(\zeta)^2}.
\]
From \eqref{Wronskian-X}, we know that \(g\) has no zeros on
\(\overline{\mathbb D}\). Consequently,
\[
H\in A^\infty(\mathbb D),
\qquad
H(0)=0,
\qquad
H'=g^{-2}\neq0.
\]
Set \(U=X(0)\). Integrating \(g^2X'=V\) gives
\[
X=U+VH.
\]
Substitution into \eqref{Wronskian-X} yields
\[
g^{-2}
=
\det(X',X)
=
g^{-2}\det(V,U),
\]
Consequently, \(\det(V,U)=1\),
and in particular \(V\neq0\).

Let
\[
\omega=e^{2\pi i/3}.
\]
Since \(g(\omega z)=g(z)\), the function
\[
\widetilde H(z)=\omega^{-1}H(\omega z)
\]
satisfies
\[
\widetilde H'(z)=H'(z),
\qquad
\widetilde H(0)=H(0)=0.
\]
Hence
\begin{equation}
\label{endpoint-H-symmetry}
H(\omega z)=\omega H(z)
\qquad\text{in }\mathbb D.
\end{equation}

For fixed \(z\in\mathbb S^1\), \eqref{X-inner} and
\eqref{endpoint-H-symmetry} give
\[
1
=
|X(\omega^kz)|^2
=
|U+\omega^kVH(z)|^2,\quad k=0,1,2
\]
Summing these identities and using
\[
1+\omega+\omega^2=0,
\]
we obtain
\[
3=|U+VH(z)|^2+|U+\omega VH(z)|^2+|U+\omega^2VH(z)|^2
=
3|U|^2+3|V|^2|H(z)|^2.
\]
Thus \(|H|\) is constant on \(\mathbb S^1\), so there exists
\(\rho\geq0\) such that
\[
|H|=\rho
\qquad\text{on }\mathbb S^1.
\]
Moreover, \(\rho>0\). Otherwise, \(H=0\) on
\(\mathbb S^1\), and by the maximum principle we have \(H\equiv0\),
contradicting \(H'\neq0\).

Define
\[
\psi=\frac{H}{\rho}.
\]
Then
\[
\psi(0)=0,
\qquad
|\psi|=1\quad\text{on }\mathbb S^1,
\qquad
\psi'=\frac{1}{\rho g^2}\neq0.
\]
Applying the maximum principle gives
\[
\psi(\mathbb D)\subset\mathbb D.
\]
The continuous extension of \(\psi\) to
\(\overline{\mathbb D}\), together with  \(\psi(\mathbb S^1)\subseteq\mathbb S^1\),
makes the map \(\psi: \mathbb D\rightarrow\mathbb D\) proper. Indeed, for every compact set \(K\Subset\mathbb D\), the set \(\psi^{-1}(K)\) is closed in \(\overline{\mathbb D}\) and disjoint from \(\mathbb S^1\), and is therefore compactly contained in $\mathbb D$.
The condition \(\psi'\neq0\) makes this proper holomorphic map unbranched, so it is a covering of \(\mathbb D\). As \(\mathbb D\) is simply connected, the covering is  one-sheeted, and hence
\[\psi\in\operatorname{Aut}(\mathbb D).\]

An automorphism of $\mathbb D$ fixing the origin is a rotation, so
\(\psi(z)=e^{i\gamma}z\) for some \(\gamma\in\mathbb R\).
Moreover, the identity
\[
\psi'(0)=\frac{1}{\rho g(0)^2}=\frac{1}{\rho b_0^2}>0
\]
forces \(\psi'(0)=e^{i\gamma}=1\), giving
\[
\psi(z)=z,
\qquad
\rho=b_0^{-2}.
\]
It follows that
\[
H(z)=b_0^{-2}z
,\quad
g^{-2}=H'= b_0^{-2}.
\]
The normalization \(g(0)=b_0>0\), together with the connectedness of
\(\mathbb D\), forces
\[
g\equiv b_0.
\]
This contradicts the assumption that \(g\) is nonconstant.
Thus \(b_2\neq0.\)
Alternative \({\rm(ii)}\) holds, and the proof is complete.

\end{proof}

\subsection{PDE form of the quadratic obstruction}

In this subsection, we finish the proof of Theorem \ref{pde-quadratic}.

The following proposition converts the spectral obstruction $b_2\neq0$ into the projective obstruction $q(0)\ne0$, thereby providing the final ingredient in the proof of Theorem \ref{pde-quadratic}.

\begin{proposition}\label{prop-q0-b2}
We have
\begin{equation*}
 q(0)=2a\frac{b_2}{b_0}.
\end{equation*}
Moreover, if \(1<a\leq2\), then
\begin{equation}\label{eq-q0-nonzero}
 q(0)\ne0
\end{equation}
for every non-affine lift $Y$.
\end{proposition}

\begin{proof}
From Proposition \ref{projective-outer}, we know that
\[
 \Phi(z)=aDh(z)=2ah_2z^2+O(z^3),
\]
\[g(z)=b_0+\sum_{n=2}^\infty b_nz^n= e^{h(z)}=e^{h_0}e^{\sum_{n=2}^\infty h_nz^n}.\]
Thus $b_0=g(0)>0$ and \( b_2=b_0h_2\).
On the other hand, since \(\Phi=z^2q\), we have
\[
 \Phi(z)=q(0)z^2+O(z^3),
\]
which yields
$$ q(0)=2ah_2=2a\frac{b_2}{b_0}.$$

We now prove \eqref{eq-q0-nonzero}.
Assume that \(Y\) is non-affine. Indeed, Theorem \ref{quadratic-obstruction} rules out a constant outer factor $g$, because that would force $Y$ to be affine. Alternative (ii) of Theorem \ref{quadratic-obstruction} therefore gives $b_2\neq0$, and the preceding identity yields
$q(0)\ne0$.
\end{proof}

We are now ready to prove  Theorem \ref{pde-quadratic} and Corollary \ref{radial-all-lambda}.
\

\textbf{\textbf{ Proof of Theorem \ref{pde-quadratic}. }}
By Theorem \ref{quadratic-obstruction}, if \(g\) is constant, then \(Y\)
is affine. Following the same argument as in the proof of Theorem \ref{radial-rigidity}, we conclude that
\(u\) is radial and has the form \eqref{R-class-large}.

If \(g\) is nonconstant, applying Proposition \ref{prop-q0-b2} and \eqref{eq-q-u} yields
\[q(0) =u_z(0)^2-u_{zz}(0)\ne0.\]

If \(u\) is radially symmetric, then
\begin{equation}\label{diff-u-once}
 u(e^{i\alpha}z)=u(z).
\end{equation}
Differentiating \eqref{diff-u-once} in \(z\) gives
\[
 u_z(e^{i\alpha}z)=e^{-i\alpha}u_z(z),
 \qquad
 u_{zz}(e^{i\alpha}z)=e^{-2i\alpha}u_{zz}(z).
\]
Hence
\[
 q(e^{i\alpha}z)=e^{-2i\alpha}q(z).
\]
If \(q(z)=\sum_{n=0}^\infty q_nz^n\), then for every \(\alpha\),
\[
 \sum_{n=0}^\infty q_ne^{in\alpha}z^n
 =e^{-2i\alpha}\sum_{n=0}^\infty q_nz^n.
\]
Since $n+2\neq0$, we derive \(q_n=0\) for every $n$, and hence
\begin{equation}\label{radial-q-zero}
 q\equiv0.
\end{equation}
Thus every radial solution satisfies \(q(0)=0\). Therefore
alternative \textup{(ii)} cannot occur in the radial case.

We now consider the case \(1+\sqrt2<\lambda\leq3\) . Suppose, for contradiction, that \(u\) is radially symmetric.
Substituting \eqref{eq-q-u} and \eqref{radial-q-zero} into \eqref{diff-p-zz} gives
\[p_{zz}=0.\]
Writing \( p(z)=P(|z|^2)\), we obtain
\[
 p_{zz}=P''(|z|^2)\bar z^2.
\]
Thus \(P''=0\), and $p$ must take the form
\begin{equation}\label{classi-p-two}
 p(z)=\alpha+\beta|z|^2.
\end{equation}
Substituting this into \eqref{p-diff-identity} yields
\[
 \alpha\beta=\frac14.
\]
Since \(\alpha=p(0)>0\), we also have
\(\beta>0\). Setting \(R=2\beta\), we obtain
\[
\alpha=\frac{1}{2R},
\qquad
\beta=\frac{R}{2}.
\]
Together with \eqref{classi-p-two} yields
$$p=e^{-u}=\frac1{2R}+\frac R2|z|^2.$$
This yields the radial form \eqref{R-class-large}. Finally, the boundary condition in \eqref{Onofri-E-L} leads to
\begin{equation}\label{lam-R-equation}
(\lambda-2)R^{2}-2R +\lambda=0.
\end{equation}

The discriminant of \eqref{lam-R-equation} is
\[
 \Delta=4\bigl(1+2\lambda-\lambda^2\bigr)
 =4\bigl(2-(\lambda-1)^2\bigr).
\]
If \(\lambda>1+\sqrt2\), then we have
\[
 \Delta<0.
\]
Hence the equation \eqref{Onofri-E-L} admits no radial
solution when \(1+\sqrt2<\lambda\leq3\).
\hfill\qedsymbol

\textbf{ Proof of Corollary \ref{radial-all-lambda}. }
As shown in the proof of Theorem~\ref{pde-quadratic}, every radial
solution is of the form
\[
u(z)=\log\frac{2R}{1+R^2|z|^2},
\qquad R>0,
\]
where \(R\) satisfies
\begin{equation}
\label{R-equation}
(\lambda-2)R^2-2R+\lambda=0.
\end{equation}
Conversely, a direct computation shows that every positive root of
\eqref{R-equation} gives a radial solution of
\eqref{Onofri-E-L}.

The discriminant of quadratic equation \eqref{R-equation} is
$$\Delta = (-2)^{2} - 4(\lambda-2)\lambda=4\bigl(2 - (\lambda-1)^{2}\bigr).$$
If \(0<\lambda<2\), then \(\Delta>0\), and
\eqref{R-equation} has exactly one positive root
\[
R=\frac{\lambda}
{1+\sqrt{1+2\lambda-\lambda^2}}.
\]
Similarly, a direct calculation gives $(ii)-(v)$.

\vskip 1cm
\noindent {\bf Acknowledgements}\\
%The author would like to thank the referee for his/her careful reading of the manuscript and many good suggestions.
This project is supported by  the National Natural Science Foundation of China (Grant No. 12571124, 12471109), Youth Innovation Team of Shaanxi Universities, and the Fundamental Research Funds for the Central Universities (Grant No. GK202402004, GK202506023).

%% bibliography--------------------------------------------------------------------
%\begin{center}

 \end{document}